\documentclass[12pt]{amsart}
\usepackage{multicol}
\usepackage{enumerate}
\usepackage{mathtools}
\usepackage{amsfonts}
\usepackage{amssymb}
\usepackage{mathrsfs}
\usepackage{graphicx}
\usepackage{url}
\usepackage{microtype}
\usepackage[margin=1in]{geometry}
\newtheorem{thm}{Theorem}[section]
\newtheorem{prop}[thm]{Proposition}

\newtheorem{cor}[thm]{Corollary}

\newtheorem{lemma}[thm]{Lemma}
\newtheorem*{thm*}{Theorem}						
\newtheorem*{prop*}{Proposition}
\newtheorem*{lemma*}{Lemma}
\newtheorem*{cor*}{Corollary}
\newtheorem*{conj*}{Conjecture}

\theoremstyle{definition}
\newtheorem{definition}[thm]{Definition}

\newtheorem{notation}[thm]{Notation}

\newcommand{\Z}{\mathbb{Z}}
\newcommand{\Q}{\mathbb{Q}}
\newcommand{\R}{\mathbb{R}}
\newcommand{\C}{\mathbb{C}}

\renewcommand{\vec}[1]{\mathbf{#1}}
\renewcommand{\P}{\mathbb{P}}

\DeclareMathOperator{\aut}{Aut}

\DeclareMathOperator{\eff}{Eff}
\DeclareMathOperator{\nef}{Nef}

\DeclareMathOperator{\pic}{Pic}
\let\O\relax
\DeclareMathOperator{\O}{\mathcal{O}}
\DeclareMathOperator{\ch}{ch}
\DeclareMathOperator{\rk}{rk}

\DeclareMathOperator{\moricone}{NE}

\DeclareMathOperator{\coh}{Coh}

\numberwithin{equation}{thm}

\usepackage{color}
\usepackage[dvipsnames]{xcolor}

\makeatletter
\let\c@theorem\c@figure
\makeatother
\begin{document}

\title[The nef cone of the Hilbert scheme of points on rational elliptic surfaces]{The nef cone of the Hilbert scheme of points on rational elliptic surfaces and the cone conjecture}

\author[J. Kopper]{John Kopper}
\address{Department of Mathematics, Statistics and CS \\ University of Illinois at Chicago, Chicago, IL, 60607}
\email{jkoppe2@uic.edu}

\thanks{During the preparation of this article the author was partially supported by NSF RTG grant DMS-1246844.}


\begin{abstract}
We compute the nef cone of the Hilbert scheme of points on a general rational elliptic surface. As a consequence of our computation, we show that the Morrison-Kawamata cone conjecture holds for these nef cones.
\end{abstract}

\maketitle

\section{Introduction}
The Morrison-Kawamata cone conjecture gives a description of the nef cone of Calabi-Yau varieties in terms of the automorphism group of the variety. A more general conjecture was formulated for klt Calabi-Yau pairs in \cite{totaro-conj} and proven in \cite{totaro} for 2-dimensional varieties. The conjecture states that there is a rational polyhedral fundamental domain for the action of the automorphism group on the nef cone. In this paper we prove that the Morrison-Kawamata cone conjecture holds for Hilbert schemes of points on rational elliptic surfaces. This fact was first shown for the underlying rational elliptic surface by Grassi-Morrison \cite{grassi-morrison}. This work is based on part of the author's Ph.D. thesis \cite{kopper-thesis}.

The heart of this paper is in giving a precise compuation of the nef cone of Hilbert schemes of points on rational elliptic surfaces. Such surfaces behave similarly to del Pezzo surfaces, for which the nef cones of Hilbert schemes have already been described \cite{bertram-coskun} \cite{7-author}. The argument in \cite{7-author} uses Bridgeland stability techniques and the Bayer-Macr\`i positivity lemma \cite{bayer-macri}. We adopt this approach and will use many of techniques developed in \cite{7-author}.

By a \emph{rational elliptic surface} we mean a smooth, rational, complex projective surface $X$ admitting a map $X \to \P^1$ whose general fiber is a smooth elliptic curve. The general such surface may be obtained as the blow-up of $\P^2$ at the nine base points of a general cubic pencil.

For any smooth surface $X$, we will denote by $X^{[n]}$ the \emph{Hilbert scheme of $n$ points on $X$} parameterizing length-$n$ subschemes of $X$. By Fogarty's theorem \cite{fogarty-smooth}, $X^{[n]}$ is smooth and irreducible of dimension $2n$. We describe $\nef(X^{[n]})$ dually by describing the cone of curves $\moricone(X^{[n]})$.

The natural map
\[
X^{[n]} \to X^{(n)}
\]
called the \emph{Hilbert-Chow morphism} is a resolution of singularities. Let $C_0$ denote the class of a curve contracted by the Hilbert-Chow morphism. Any curve $C \subset X$ admitting a $g^1_n$ induces a rational curve $C_{[n]} \subset X^{[n]}$ whose points parameterize the fibers of the corresponding map $C \to \P^1$. Our main theorem can then be stated as follows.

\begin{thm}\label{thm:main_thm}
Let $X$ be a general rational elliptic surface and $n \geq 3$ an integer. Then the cone of curves $\moricone(X^{[n]})$ is spanned by the classes $E_{[n]}$ for all $(-1)$-curves $E \subset X$, the curve $F_{[n]}$, where $F$ is the class of an elliptic fiber, and the class $C_0$.
\end{thm}

There are infinitely many $(-1)$-curves on $X$, but they fall into single orbit of the action of the Weyl group \cite{looijenga}. Furthermore, there is a rational polyhedral fundamental domain for the action of $\aut(X)$ on $\moricone(X)$ \cite{grassi-morrison}. Assuming this description of $\moricone(X)$, Theorem \ref{thm:main_thm} describes $\moricone(X^{[n]})$ completely. As a consequence, we show that there is a fundamental domain for the action of the automorphism group on the nef cone and deduce that the Morrison-Kawamata cone conjecture also holds for $\nef(X^{[n]})$ (Corollary \ref{cor:cone_conj}).

\subsection*{Acknowledgments} The author is grateful to Izzet Coskun and Tim Ryan for their many helpful conversations.

\section{Preliminaries} We will work over the field of complex numbers. In this section we collect a number of facts required for our calculation of the nef cone. Let $X$ be a smooth projective surface with $q(X) = 0$. Then by \cite{fogarty} we have $\rho(X^{[n]}) = \rho(X) + 1$ and the Picard group of $X^{[n]}$ can be described as follows. Given a divisor $D$ on $X$, define the divisor $D^{[n]}$ on $X^{[n]}$ by
\[
D^{[n]} = \left\{Z \in X^{[n]}: D \cap Z \neq \varnothing \right\}.
\]
Let $h:X^{[n]} \to X^{(n)} = (X \times X \times \cdots \times X)/\mathfrak{S}_n$ denote the Hilbert-Chow morphism sending a length $n$ scheme to its support. The exceptional locus of $h$ is a divisor which we denote $B$ and it parameterizes nonreduced schemes. Then we have
\[
\pic(X^{[n]}) \cong \pic(X) \oplus \Z \frac{B}{2},
\]
where $\pic(X)$ embeds into $\pic(X^{[n]})$ by $D \mapsto D^{[n]}$.

We introduce a few curve classes on $X^{[n]}$. We denote by $C_0$ the class of a curve contracted by $h$. If $C \subset X$ is a curve that admits a $g^1_n$, i.e., a degree $n$ map to $\P^1$, then the fibers of this map induce a rational curve $C_{[n]}$ on $X^{[n]}$. The following intersection numbers are standard:

\begin{center}
\bgroup
\def\arraystretch{1.5}
\begin{tabular}{c|cc}
& $D^{[n]}$ & $B$\\
\hline
$C_{[n]}$ & $C \cdot D$ & $2g(C) - 2 + 2n$\\
$C_0$ & $0$ & $-2$
\end{tabular}
\egroup
\end{center}

Note that any divisor on $X^{[n]}$ that has nonnegative intersection with $C_0$ and $C_{[n]}$ for every effective curve $C$ must have nonpositive coefficient in $B$.

\subsection{Background on Bridgeland stability}
The technical core of this paper uses Bridgeland stability conditions for derived categories. We will not need the full strength of Bridgeland's results, so we refer the reader to Bridgeland's papers \cite{bridgeland-1}\cite{bridgeland-2} for more details. We also suggest Huizenga's survey \cite{huizenga-survey} for more background on Bridgeland stability conditions and their relationship to moduli spaces of sheaves.

Let $X$ be a smooth projective surface and fix a polarization $A \in \pic(X) \otimes \Q$. Let $P \in \pic(X) \otimes \Q$ be any $\Q$-divisor. The \emph{($P$-)twisted Chern character} is defined as $\ch^P = \exp(-P) \ch$.

We define the $\mu_{A,P}$-slope to be the function
\[
\mu_{A,P} = \begin{cases}
\displaystyle \frac{A \cdot \ch_1^P}{A^2 \ch_0^P} & \ch_0^P \neq 0\\[10pt]
+\infty & \ch_0^p = 0.
\end{cases}
\]

The \emph{$(A,P)$-twisted discriminant} is defined as
\[
\Delta_{A,P} = \frac{1}{2}\mu_{A,P}^2 - \frac{\ch_2^P}{A^2 \ch_0^P}.
\]

Arcara-Bertram \cite{arcara-bertram} construct a family of stability conditions called the \emph{$(A,P)$-slice}. Explicitly, let $s \in \R$ and define a category
\[
\mathcal{A}_s = \langle\mathcal{F}_s[1], \mathcal{Q}_s\rangle,
\]
where
\begin{align*}
\mathcal{F}_s &= \{F \in \coh(X) : F \text{ is torsion-free and } \mu_{A,P}(F') \leq s \text{ for all proper subsheaves }F' \text{ of } F \}\\
\mathcal{Q}_s &= \{Q \in \coh(X) : Q \text{ is torsion or } \mu_{A,P}(Q') > s \text{ for all quotients } Q' \text{ of } Q \}.
\end{align*}
For $s \in R$ and $t \in \R_{>0}$, the \emph{central charge} $Z_{s,t}$ is defined as the homomorphism $Z: K_0(X) \to \C$ given by
\[
Z(\vec{v}) = -\ch_2^{P+sA} + \frac{t^2A^2}{2}\ch_0^{P+sA}+iA\cdot \ch_1^{P+sA}.
\]
The \emph{$\mu_{s,t}$-slope} of an object $E \in \mathcal{A}_s$ is defined as the number
\[
\mu_{s,t}(E) = -\frac{\Re Z_{s,t}(E)}{\Im Z_{s,t}(E)}.
\]
Then Arcara-Bertram show in \cite{arcara-bertram} that the pair $(\mathcal{A}_s, Z_{s,t})$ is a Bridgeland stability condition. Moreover, for a fixed Chern character $\vec{v} \in K_0(X)$, the upper half-plane $\{(s,t):t>0\}$ (identified with the $(A,P)$-slice) admits a wall-and-chamber decomposition such that stable objects of Chern character $\vec{v}$ cannot destabilize unless a wall is crossed. Maciocia showed in \cite{maciocia} that these walls are either the vertical line $s=\mu_{A,P}(\vec{v})$ or nested semicircles. A semicircular wall $W_1$ is contained in a semicircular wall $W_2$ if and only if the center of $W_1$ is to the left of the center of $W_2$.

The following lemma collects some standard facts about the stability conditions in the $(A,P)$-slice.

\begin{lemma} With notation as above, we have the following.
\begin{enumerate}
\item If $E$ is a sheaf in $\mathcal{A}_s$ and $F \to E$ is a destabilizing object in $\mathcal{A}_s$, then $F$ is a sheaf.
\item For $t\gg 0$, the Bridgeland moduli space of $(\mathcal{A}_s,Z_{s,t})$-semistable objects with Chern character $\vec{v}$ equals the $(A,P)$-twisted moduli space of Gieseker semistable sheaves with Chern character $\vec{v}$.
\item Given objects $E$ and $F$ in $\mathcal{A}_s$, the numerical wall $W(E,F)$ in the $(A,P)$-slice consisting of stability conditions for which $E$ and $F$ have the same $\mu_{s,t}$-slope has center $s_0$ and radius $\rho$ satisfying
\begin{align*}
 s_0 &= \frac{1}{2}(\mu_{A,P}(E) + \mu_{A,P}(F)) - \frac{\Delta_{A,P}(E) - \Delta_{A,P}(F)}{\mu_{A,P}(E) - \mu_{A,P}(F)}\\
 \rho^2 &= (\mu_{A,P}(E) - s_0)^2 - 2\Delta_{A,P}(E).
\end{align*}
In particular, if $I_Z$ is the ideal sheaf of a length $n$ scheme and $E = L \otimes I_{Z'}$ for some line bundle $L$ and 0-dimensional scheme $Z'$ of length $m$, then the center $s_0$ of $W(E, I_Z)$ is given by
\[
s_0 = \frac{n-m+\frac{L^2}{2}-\frac{L\cdot P}{2}}{-L \cdot A}.
\]
\end{enumerate}
\end{lemma}

Of particular interest to us is the case $\vec{v}=(1,0,-n)$ for $n \in \Z_{>0}$. In this case, the Gieseker moduli space is isomorphic to the $X^{[n]}$ via the map sending an ideal sheaf $I_Z$ to the scheme $Z$. The largest wall in the $(A,P)$-slice for $\vec{v}$ is called the \emph{Gieseker wall}, and determining it precisely is the central computation in the proof of Theorem \ref{thm:main_thm}.

\subsection{The Weyl action} Let $X$ be a general rational elliptic surface. We identify $X$ with the blow-up $\pi:X \to \P^2$ of $\P^2$ at the nine base points of a general cubic pencil. The $\pi$-exceptional divisors $E_1, \dots, E_9$ on $X$ are disjoint sections of the elliptic fibration $X \to \P^1$. The Picard group of $X$ is given by
\[
\pic(X) \cong \Z H \oplus \Z E_1 \oplus \cdots \Z E_9,
\]
where $H$ is the pullback of the hyperplane class via $\pi$. The following intersection numbers are standard:
\[
H^2 = 1, \qquad H \cdot E_i = 0 \text{ for all }i, \qquad E_i^2 = -1, \qquad E_i \cdot E_j = 0 \text{ for } i \neq j.
\]
The canonical class of $X$ is $K_X = -3H + E_1 + \cdots + E_9$. We denote by $F$ the class of an elliptic fiber and note that $F = -K_X$. In particular, $K_X$ is antieffective.

\begin{prop}[\protect{\cite{looijenga}}]
The cone of curves $\moricone(X)$ is generated by the fiber class $F$ together with the classes of all $(-1)$-curves.
\end{prop}

It is well-known that there are infinitely many $(-1)$-curves on $X$, but they fall into a single orbit of the action of the Weyl group as we explain below.

\begin{definition}
The \emph{root lattice} on $X$ is the orthogonal complement of $F$ in $N^1(X)$ with respect to the intersection pairing. It has an integral basis
\[
B = \{ E_1 - E_2, E_2-E_3, \dots, E_8-E_9, H-E_1-E_2-E_3\}.
\]
\end{definition}

Given an element $\beta \in B$, there is an automorphism $s_\beta:N^1(X) \to N^1(X)$ defined by
\[
s_\beta(D) = D + (D \cdot \beta)D.
\]
The group generated by all reflections for $\beta \in B$ is called the \emph{Weyl group of $X$}. Note that the action of the Weyl group preserves intersection numbers: since $\beta^2 = -2$ for all $\beta \in B$, we have
\[
s_\beta(D) \cdot s_\beta(D') = D \cdot D' + \beta^2(D \cdot \beta)(D' \cdot \beta) + 2(D\cdot \beta)(D' \cdot \beta) = D \cdot D'.
\]

It is also clear that the Weyl group fixes the fiber class $F$.

\begin{lemma}[\protect{\cite{looijenga}\cite{7-author}}]\label{lemma:orbits}
Let $X$ be a general rational elliptic surface. Then
\begin{enumerate}
	\item The Weyl group of $X$ acts transitively on the set of $(-1)$-curves in $\moricone(X)$.
	\item The Weyl group action on the set of extremal rays of $\nef(X)$ has three orbits, represented by the classes $F$, $H$, and $H-E_1$.
\end{enumerate}
\end{lemma}
\begin{proof}
Statement (1) is in \cite{looijenga}. For statement (2), let $D$ be an extremal nef divisor. Then $D$ is orthogonal to a face of $\moricone(X)$, so either $D \cdot F = 0$ or $D \cdot E = 0$ for some $(-1)$-curve $E$. If $D \cdot F = 0$, then $D$ is parallel to the fiber class, hence is in the orbit of $F$.

On the other hand suppose $D \cdot E = 0$ for some $(-1)$-curve $E$. By (1) there is an element of the Weyl group taking $E$ to $E_9$. Let $X \to Y$ be the blow-down of $E_9$. Then $Y$ is the degree 1 del Pezzo surface isomorphic to the blow-up of $\P^2$ at 8 general points and $D = \pi^\ast D'$ for an extremal nef divisor $D'$ on $Y$. By \cite[Proposition 5.2]{7-author}, $D'$ is either in the orbit of $H$ or $H-E_1$ under the action of the Weyl group for $Y$.
\end{proof}

\section{The nef cone} Let $X$ be a general rational elliptic surface. The calculation of $\nef(X^{[n]})$ follows the method of \cite{7-author} and proceeds in roughly two steps. First, we bound $\nef(X^{[n]})$ by describing an \emph{a priori} larger cone $\Lambda$ that must contain it. Second, we show that every ray of $\Lambda$ is nef. To demonstrate their nefness, we exhibit them as Bayer-Macr\`i divisors. This relies on a choice of polarization and twisting divisor.

Let $\Lambda\subset N^1(X^{[n]})$ be the cone spanned by all divisor classes that have nonnegative intersection with $F_{[n]}$, $C_0$, and $E_{[n]}$ for all $(-1)$-curves $E$. 

\begin{lemma}
The cone $\Lambda$ is contained in the cone spanned by $\nef(X) \subset \nef(X^{[n]})$ and the class
\[
(n-1)F^{[n]} - \frac{1}{2}B.
\]
\end{lemma}
\begin{proof}
Let $D\in \Lambda$. Write $D = C^{[n]} - a B$ where $C \in \nef(X)$. If $a =0$, then $D \in \nef(X)$, so we assume $a > 0$. Scaling by $\Q$, we may assume $a = \frac{1}{2}.$

We show that the class $C-(n-1)F$ is nef on $X$. It will then follow that
\[
D = \left[C^{[n]} - (n-1)F^{[n]}\right] + \left[(n-1)F^{[n]} - \frac{1}{2}B\right]
\] exhibits $D$ in the desired cone. We have
\[
(C - (n-1)F) \cdot F = C \cdot F \geq 0
\]
because $C$ is nef. If $E$ is a $(-1)$-curve, then
\[
(C - (n-1)F)\cdot E = C \cdot E - (n-1).
\]
Since $0 \leq D \cdot E_{[n]} = C \cdot E - (n-1)$, it follows that $(C-(n-1)F)\cdot E \geq 0$.

On the other hand, by adjunction we have $F \cdot E = 1$ for all $(-1)$-curves $E$, and therefore that
\[
\left((n-1)F^{[n]} - \frac{1}{2}B\right) \cdot E_{[n]}=0,
\]
and
\[
\left((n-1)F^{[n]} - \frac{1}{2}B\right) \cdot C_0=1,
\]
proving the claim.
\end{proof}
Note that $((n-1)F^{[n]} - \frac{1}{2}B) \cdot F_{[n]} = -n$ and so $(n-1)F^{[n]} - \frac{1}{2}B$ is not nef. If $D \in \Lambda$ spans an extremal ray, then there is an extremal ray $C$ of $\nef(X)$ such that $D$ is the unique $F_{[n]}$-orthogonal ray in the plane spanned by $(n-1)F^{[n]} - \frac{1}{2}B$ and $C^{[n]}$.

\begin{notation}
For $C \in \nef(X)$, we will write $\varepsilon(C)\in N^1(X^{[n]})$ to denote the unique $F^{[n]}$-orthogonal ray constructed above.
\end{notation}

By Lemma \ref{lemma:orbits}, the extremal ray $C$ of $\nef(X)$ is in the Weyl orbit of $F$, $H$, or $H-E_1$. If $C$ is in the orbit of $F$, then $C=F$ and $D = F^{[n]}$ which is certainly nef. Thus up to the Weyl action we may assume $C= H$ or $C=H-E_1$. Computing the $F_{[n]}$-orthogonal divisor described above, we have two cases:
\begin{align*}
\varepsilon(H) &= (n-1)F^{[n]}+\frac{n}{3}H^{[n]}-\frac{1}{2}B \\
\varepsilon(H-E_1) &= (n-1)F^{[n]}+\frac{n}{2}\left(H^{[n]}-E_1^{[n]}\right)-\frac{1}{2}B.
\end{align*}

To prove the theorem it suffices to show that the above classes are nef because it will follow that $\Lambda \subset \nef(X^{[n]})$ (and the other containment is obvious). In order to do so, we invoke the following theorem from \cite{7-author}.

\begin{thm*}[\protect{\cite[Proposition 3.8]{7-author}}]
Let $X$ be a smooth projective surface with irregularity zero, $A$ an ample divisor, and $P$ an antieffective $\Q$-divisor. Let $\sigma$ be a stability condition lying on a numerical wall with center $s_W$ in the Gieseker chamber in the $(A,P)$-slice corresponding to the Chern character $\vec{v} = (1,0,-n)$. Then the divisor
\[
\frac{1}{2}K_X^{[n]}-s_WA^{[n]}-P^{[n]}-\frac{1}{2}B
\]
is nef on $X^{[n]}$.
\end{thm*}

In proving Theorem \ref{thm:main_thm} we will always choose $P=-F$. We will work with two polarizations:
\begin{align*}
A_1 &= \frac{n}{3}H + \left(n - \frac{3}{2}\right) F\\
A_2 &= \frac{n}{2}(H-E_1) + \left(n - \frac{3}{2}\right) F.
\end{align*}

\begin{lemma}
The classes $A_1$ and $A_2$ are ample on $X^{[n]}$ for $n \geq 2$.
\end{lemma}
\begin{proof}
We first consider $A_1$. Since $\moricone(X)=\overline{\moricone(X)}$ is spanned by $F$ and the $(-1)$-curves $E$, it suffices to show that $A_1$ has positive intersection with those classes. We have
\[
A_1 \cdot F = n >0.
\]
If $E$ is a $(-1)$-curve, then either $E=E_i$ for some $i$ or $E=aH - \sum_{i=1}b_iE_i$ for $a>0$, $b_i \geq 0$ for all $i$. In the first case,
\[
A_1 \cdot E_i = n- \frac{3}{2} > 0.
\]
In the second case,
\[
A_1 \cdot E = \frac{n}{3}a+n - \frac{3}{2} > 0.
\]
Finally,
\[
A_1^2 = \frac{10n^2}{9}-\frac{3n}{2} > 0.
\]
An analogous calculation shows that $A_2$ is ample.
\end{proof}

We are now ready to prove the main theorem. The remainder of the argument is to compute the Gieseker wall in the $(A_1,-F)$-slice and in the $(A_2, -F)$-slice. We will see that for $\sigma$ on this wall, we have
\[
\frac{1}{2}K_X^{[n]}-s_WA_i^{[n]}+F^{[n]}-\frac{1}{2}B =
\begin{cases}
\varepsilon(H) & \text{if } i=1,\\
\varepsilon(H-E_1) & \text{if }i=2.
\end{cases}
\]
By the theorem of \cite{7-author}, it will follow that $\varepsilon(H)$ and $\varepsilon(H-E_1)$ are nef.

\begin{proof}[Proof of Theorem \ref{thm:main_thm}]
We first compute the Gieseker wall in the $(A_1,-F)$-slice for the Chern character $\vec{v}=(1,0,-n)$. Let $Z$ be a length $n$ subscheme of $X$ consisting of distinct points lying on a single elliptic fiber. Then there is an injective map of sheaves
\[
\O_X(-F) \to I_Z.
\]
The corresponding wall $W(\O_X(-F),I_Z)$ in the $(A_1,-F)$-slice has center
\[
s_W = \frac{n + \frac{F^2}{2}}{-F \cdot A_1} = -1.
\]
We wish to show that this is the Gieseker wall, i.e., that there is no larger wall destabilizing an ideal sheaf of the same character. Since the walls for $\vec{v}$ are nested semicircles, it suffices to show that the center of any actual wall is larger than $-1$. Suppose for a contradiction that there is a map $U \to I_{Z}$ corresponding to a larger wall that destabilizes some ideal sheaf of a length $n$ scheme $Z \subset X$. We distinguish two cases: $\rk(U) =1$ and $\rk(U) > 1$.

In the first case, we may write $U = L \otimes I_{Y}$, where $L$ is a line bundle and $I_Y$ is some ideal sheaf. Then $L$ and $U$ have the same slope and the center of the wall $W(U, I_{Z})$ is given by
\[
\frac{n-\ell(Y)+\frac{L^2}{2}-\frac{L\cdot P}{2}}{L \cdot A_1},
\]
whereas the center of $W(L,I_{Z})$ is given by
\[
\frac{n+\frac{L^2}{2}-\frac{L\cdot P}{2}}{L \cdot A_1}.
\]
Evidently $W(L,I_{Z})$ is larger than $W(U,I_{Z})$, and so we may assume $U = L$ without loss of generality.
If $-L = E_i$ for some $i$, then the wall $W(L, I_Z)$ has center
\[
-\frac{n-1}{n-\frac{3}{2}} > -1,
\]
so we assume $-L$ is of the form $aH - \sum_{i=1}^9 b_i E_i$ for $a>0$ and $b_i \geq 0$ for all $i$. By \cite[Prop 3.5]{7-author} we may also assume that $L$ is antieffective and that
\[
-L \cdot A_1 < A_1 \cdot F = n.
\]
If we have furthermore that $-L \cdot F \geq 2$, then
\[
-L \cdot A_1 \geq \frac{n}{3}-L\cdot H + 2n-3>n
\]
when $n\geq 3$. We may therefore assume that $-L \cdot F \leq 1$. If $-L\cdot F = 0$ then $L$ is parallel to $F$ and we are done, so assume that $-L\cdot F = 1$. We have
\[
-L \cdot A_1 = \frac{an}{3} + \left(n-\frac{3}{2}\right).
\]
Thus if $-L \cdot A_1 < n$, we must have $a=0$ or $a=1$. If $a=0$, then $-L = E_i$ for some $i$, a case we have already considered. There are three effective classes with $a=1$:
\begin{enumerate}
	\item $-L = H$, or
	\item $-L = H-E_i$ for some $i$, or
	\item $-L = H-E_i -E_j$ for some $i \neq j$.
\end{enumerate}
Since we have assumed that $-L \cdot F = 1$, the only possibility is $-L = H-E_i-E_j$ for some $i\neq j$. In this case, the center of $W(L, I_Z)$ equals
\[
-\frac{n-1}{\frac{4}{3}n-\frac{3}{2}} > -1.
\]
Thus $-L=F$ corresponds to the largest rank 1 wall.

The second case to consider is $\rk(U) \geq 2$. By \cite[Cor. 3.2]{7-author}, the radius $\rho'$ of the wall $W(U,I_Z)$ must then satisfy
\[
(\rho')^2 \leq \frac{2nA_1^2 + (A_1 \cdot -F)^2 - A_1^2F^2}{8(A_1^2)^2}.
\]
On the other hand, the radius $\rho$ of $W(\O_X(-F),I_Z)$ satisfies
\[
\rho^2 = (-1-\mu_{A_1,-F}(I_Z))^2 - 2\Delta_{A_1,-F}(I_Z) = 1+\frac{3n}{A_1^2}.
\]
Using that $F^2 = 0$ and $A_1 \cdot F=-n$, we see that for $n\geq 3$ we have $\rho^2 > (\rho')^2$. Thus no such $U$ gives a larger actual wall and so $W(\O_X(-F),I_Z)$ is the Gieseker wall. By \cite[Prop. 3.8]{7-author}, the class
\[
\frac{1}{2}K_X^{[n]} - s_WA_1^{[n]} + F^{[n]} - \frac{1}{2}B = \varepsilon(H)
\]
is nef.

The proof of the nefness of $\varepsilon(H-E_2)$ is similar. We use the polarization $A_2$ and twisting divisor $-F$. Again, by arranging the $n$ points of $Z$ on an elliptic fiber we get a map $\O_X(-F) \to I_Z$ and a corresponding wall with center $s_W = -1$. We show again that there is no larger wall in the $(A_2, -F)$-slice by separately considering destabilizing sheaves $U$ with $\rk(U) =1$ and $\rk(U) > 1$.

As before, if $\rk(U) = 1$, then we may assume without loss of generality that $U = L$ is an antieffective line bundle. If $-L = E_i$ for some $i$, then the wall $W(L,I_Z)$ has center

\[
\begin{cases}
		\displaystyle-\frac{2}{3} & \text{if } i=1\\[10pt]
		\displaystyle-\frac{n- 1}{n-\frac{3}{2}} & \text{if } 2 \leq i \leq 9.
\end{cases}
\]

In either case, $W(L,I_Z)$ is smaller than $W(\O_X(-F), I_Z)$. We therefore assume that $-L = aH - \sum_{i=1}^9 b_iE_i$ with $a>0$ and $b_i \geq 0$ for all $i$. We may again assume furthermore by \cite[Prop. 3.5]{7-author} that
\[
-L \cdot A_2 < A_2 \cdot F = n.
\]
If $-L \cdot F \geq 2$, then
\[
-L \cdot A_2 \geq 2n-\frac{3}{2} > n,
\]
so we assume $-L \cdot F = 1$ as above. Then,
\[
-L \cdot A_2 = \frac{n}{2}(a-b_1)+n-\frac{3}{2}.
\]
Thus if $-L \cdot A_2 > n$, we must have $a=b_1$. The only effective curve with class $aH - aE_1 - \cdots$ is a union of $a$ lines through the blow-down of $E_1$. Since we have also assumed that $-L \cdot F = 1$, we have
\[
2a = 1 + \sum_{i=2}^9 b_i.
\]
But $a$ lines through a fixed point can only pass through $a$ additional points (counting multiplicity), and so
\[
a\leq \sum_{i=2}^9 b_i.
\]
Thus $a=1$ and consequently $-L = H-E_1 - E_i$ for some $i>1$. The corresponding wall has center
\[
- \frac{n-1}{n - \frac{3}{2}} > -1.
\]
We conclude that $W(\O_X(-F), I_Z)$ is the largest wall corresponding to a rank 1 destabilizing object.

If $\rk(U) \geq 2$, then we again apply \cite[Cor. 3.2]{7-author} to see that the radius of $W(U,I_Z)$ must be smaller than the radius of $W(\O_X(-F), I_Z)$, and thus $W(\O_X(-F),I_Z)$ is the Gieseker wall and $\varepsilon(H-E_1)$ is nef.
\end{proof}

\subsection{The cone conjecture}
Let $X$ be a $\Q$-factorial variety and $\Delta$ an effective $\R$-divisor on $X$. We call $(X, \Delta)$ a \emph{klt Calabi-Yau pair} if $(X,\Delta)$ is klt and $K_X + \Delta$ is numerically trivial. Let $\aut(X,\Delta)$ denote the group of automorphisms of $X$ that preserve $\Delta$.

\begin{conj*}[Kawamata-Morrison cone conjecture]\label{conj:cone}
Let $(X,\Delta)$ be a klt Calabi-Yau pair. Then the number of $\aut(X,\Delta)$-equivalence classes of faces of the cone $\nef(X) \cap \eff(X)$ corresponding to birational contractions or fiber space structures is finite and there exists a fundamental domain $\Pi$ for the action of $\aut(X,\Delta)$ on $\nef(X)\cap \eff(X)$.
\end{conj*}

The cone conjecture was shown by Totaro for 2-dimensional klt Calabi-Yau pairs \cite{totaro}. Grassi-Morrison exhibited in \cite{grassi-morrison} an explicit fundamental domain when $X$ is a general rational elliptic surface. We will use the results of this section and the result of Grassi-Morrison to show that the conjecture holds for the pair $(X^{[n]},F^{[n]})$. The main observation is that enough elements of the Weyl group correspond to actual automorphisms of $X$.

Fix $E_1$ as the zero-section on each fiber of $X$. Given any other section $E$, addition by $E-E_1$ (in the group law of the fiber) gives an automorphism of each fiber that extends to an automorphism of $X$. We call such automorphisms \emph{translations}.

\begin{prop}[\protect{\cite{grassi-morrison}}]
The group of translations is contained in the Weyl group of $X$ and there is a fundamental domain for the action of the translation group on $\nef(X)$.
\end{prop}
An automorphism $\phi$ of $X$ extends naturally to an automorphism of $X^{[n]}$ by sending an ideal sheaf $I_Z$ to its pullback $\phi^\ast(I_Z)$. In particular, the translations of $X$ induce $F^{[n]}$-preserving automorphisms of $X^{[n]}$ that we will also call \emph{translations}.

\begin{cor}\label{cor:cone_conj}
The Morrison-Kawamata cone conjecture holds for the pair $(X^{[n]},F^{[n]})$.
\end{cor}
\begin{proof}
Let $\Pi$ denote the fundamental domain for $\nef(X)$ constructed by Grassi-Morrison. We construct a domain $\Psi$ as follows. We consider $\Pi$ as a subset of $\nef(X^{[n]})$ in the usual manner. Let $\Psi$ be the intersection of $\nef(X^{[n]})$ with the cone spanned by $\Pi$ and the class $(n-1)F^{[n]}-\frac{1}{2}B$.

We have seen that $\nef(X^{[n]})$ is contained in the cone spanned by $\nef(X)$ and $(n-1)F^{[n]} - \frac{1}{2}B$. To see that $\Psi$ is rational polyhedral, note that $\Psi$ is precisely the cone spanned by $\Pi$ and the set of all $\varepsilon(C)$ for all extremal edges $C$ of $\Pi$, where $\varepsilon(C)$ is the unique $F^{[n]}$-orthogonal ray in the plane $\langle C, (n-1)F^{[n]}- \frac{1}{2}B \rangle$ constructed earlier. Since $\Pi$ is rational polyhedral, it follows that $\Psi$ is too.

Suppose $D \in \nef(X^{[n]})$. Let $p(D)$ denote its projection to the $B=0$ hyperplane. Then $p(D) \in \phi(\Pi)$ for some translation $\phi \in \aut(X^{[n]},F)$, and so $D$ is in the cone spanned by $\phi(\Pi)$ and $(n-1)F^{[n]} - \frac{1}{2}B$, and thus $D \in \phi(\Psi)$. It follows that
\[
\nef(X^{[n]}) = \bigcup_{\phi \in \aut(X^{[n]},F^{[n]})} \phi(\Psi).
\]

\end{proof}

\bibliographystyle{plain}

\end{document}